\documentclass[10pt]{amsart}
\usepackage{amsmath}
\usepackage{amssymb}
\usepackage{amsfonts}
\usepackage{pxfonts}
\usepackage[OT2,T1]{fontenc}
\usepackage[
textwidth=16.0cm, 
textheight=22.0cm,
hmarginratio=1:1,
vmarginratio=1:1]{geometry}
\usepackage{pifont}
\usepackage[english]{babel}
\usepackage{graphicx}
\usepackage{mathbbol}

\newtheorem{thm}{Theorem}[subsection]
\newtheorem{cor}[thm]{Corollary}

\theoremstyle{definition}

\theoremstyle{remark}
\newtheorem{rem}[thm]{Remark}

\numberwithin{equation}{subsection}
\numberwithin{figure}{subsection}

\newcommand{\diff}{\mathrm{d}}
\newcommand{\C}{{\mathbb C}}
\newcommand{\R}{{\mathbb R}}

\newcommand{\Hyp}{{\mathbb H}}

\newcommand{\Z}{{\mathbb Z}}

\newcommand{\imag}{\mathrm{i}}
\newcommand{\e}{\mathrm{e}}

\newcommand{\ordo}{\mathrm{o}}

\newcommand{\vol}{\mathrm{vol}}

\DeclareMathOperator{\im}{Im}

\begin{document}

%
\title{Fourier uniqueness in $\R^4$}

\author{Andrew Bakan}
\address{
Bakan: Institute of Mathemtics\\
National Academy of Sciences of Ukraine
\\
Kiev 01601\\
Ukraine}
\email{andrew.g.bakan@gmail.com}

\author{Haakan Hedenmalm}
\address{
Hedenmalm: Department of Mathematics\\
KTH Royal Institute of Technology\\
S--10044 Stockholm\\
Sweden}
\email{haakanh@kth.se}

\author{Alfonso Montes-Rodr\'\i{}guez}
\address{
Montes-Rodr\'\i{}guez: Department of Mathematical Analysis\\
University of Sevilla\\
Seville\\
Spain}
\email{amontes@us.es}

\author{Danylo Radchenko}
\address{Radchenko:
Department of Mathematics\\
ETHZ\\
R\"amistrasse 101\\
CH-8092 Z\"urich\\
Switzerland}
\email{danradchenko@gmail.com}

\author{Maryna Viazovska}
\address{
Viazovska: Institute of Mathematics\\
EPFL\\
CH-1015 Lausanne\\
Switzerland}
\email{viazovska@gmail.com}

 
\thanks{This research was supported by Vetenskapsr\aa{}det (VR)}
 
\begin{abstract} 
We show an interrelation between the uniqueness aspect of the recent
Fourier interpolation formula of Radchenko and Viazovska and the Heisenberg
uniqueness for the Klein-Gordon equation and the lattice-cross of critical
density,
studied by Hedenmalm and Montes-Rodr\'\i{}guez. This has been known since 2017.
\end{abstract}

\maketitle

\section{Introduction} 

\subsection{Basic notation in the plane}
\label{subsec-1.1}
We write $\Z$ for the integers, $\Z_+$ for the positive integers, $\R$ 
for the real line, and $\C$ for the complex plane. We write $\Hyp$
for the upper half-plane $\{\tau\in\C:\,\im\tau>0\}$. Moreover, we let
$\langle\cdot,\cdot\rangle_d$ denote the Euclidean inner product of $\R^d$.

\subsection{The Fourier transform of radial functions}

For a function $f\in L^1(\R^d)$, we consider its Fourier transform (with
$x=(x_1,\ldots,x_d)$ and $y=(y_1,\ldots,y_d)$)
\[
\hat f(y):=\int_{\R^d}\e^{-\imag2\pi\langle x,y\rangle_d}f(x)\diff
\vol_d(x),
\quad \diff\vol_d(x):=\diff x_1\cdots\diff x_d.
\]
If $f$ is radial, then $\hat f$ is radial too. A particular example of
a radial function is the Gaussian
\begin{equation}
G_\tau(x):=\e^{\imag \pi\tau|x|^2},
\label{eq:Gauss}
\end{equation}
which decays nicely provided that $\im \tau>0$, that is, when $\tau\in\Hyp$.
The Fourier transform of a Gaussian is another Gaussian, in this case
\begin{equation}
\hat G_\tau(y):=\bigg(\frac{\tau}{\imag}\bigg)^{-d/2}\e^{-\imag \pi|y|^2/\tau}
=\bigg(\frac{\tau}{\imag}\bigg)^{-d/2} G_{-1/\tau}(y),
\label{eq:GL}
\end{equation}
Here, it is important that $\tau\mapsto-1/\tau$ preserves hyperbolic space
$\Hyp$. 
In the sense of distribution theory, the above relationship extends to
boundary points $\tau\in\R$ as well. We now consider the relationship
\begin{equation}
\Phi(x):=\int_\R G_\tau(x)\phi(\tau)\diff \tau=
\int_\R\e^{\imag\pi\tau|x|^2}\phi(\tau)\diff \tau,\qquad x\in\R^d.
\label{eq:Phirel}
\end{equation}
In terms of the Fourier transform, the relationship reads
\[
\Phi(x)=\hat \phi_1\bigg(-\frac{|x|^2}{2}\bigg),
\]
where the subscript signifies that we are dealing with the Fourier transform
on $\R^1$. This tells us that $\Phi$ is radial, but pretty arbitrary, if, say,
$\phi\in L^1(\R)$.
In view of the functional identity \eqref{eq:Gauss},
the Fourier transform of the radial function $\Phi$ equals
\begin{equation}
\hat\Phi(y):=\int_\R \hat G_\tau(y)\phi(\tau)\diff \tau=
\int_\R\bigg(\frac{\tau}{\imag}\bigg)^{-d/2} G_{-1/\tau}(y)\phi(\tau)\diff \tau
=\int_\R\bigg(\frac{\tau}{\imag}\bigg)^{-d/2} \e^{-\imag\pi |y|^2/\tau}
\phi(\tau)\diff \tau.
\label{eq:Phirel2}
\end{equation}
We now rewrite the relationships \eqref{eq:Phirel} and
\eqref{eq:Phirel2} using integration by parts. If $\phi$ is a
tempered test function,
integration by parts applied to \eqref{eq:Phirel} gives that
\begin{equation}
\Phi(x)=\frac{\imag}{\pi |x|^2}\int_\R\e^{\imag\pi\tau |x|^2}\phi'(\tau)
\diff \tau,\qquad x\in\R^d\setminus\{0\}.
\label{eq:IBP1}
\end{equation}
A similar application of integration by parts to \eqref{eq:Phirel2} gives
that
\begin{equation}
\hat\Phi(y)=\frac{\imag}{\pi|y|^2}
\int_\R\bigg(\frac{\tau}{\imag}\bigg)^{(4-d)/2}\phi(\tau)
\partial_\tau\e^{-\imag\pi|y|^2/\tau}\diff \tau
=\frac{1}{\imag\pi|y|^2}\int_\R\partial_\tau
\bigg\{\bigg(\frac{\tau}{\imag}\bigg)^{(4-d)/2}\phi(\tau)\bigg\}
\e^{-\imag\pi|y|^2/\tau}\diff\tau,
\label{eq:IBP2}
\end{equation}
where $y\in\R^d\setminus\{0\}$,
and we need to be a little careful around $\tau=0$ unless $d\in\{0,2,4\}$.
We now \emph{restrict to $d:=4$}, so that \eqref{eq:IBP2} simplifies to
\begin{equation}
\hat\Phi(y)=
\frac{1}{\imag\pi|y|^2}\int_\R
\phi'(\tau)\e^{-\imag\pi|y|^2/\tau}\diff\tau,\qquad y\in\R^4\setminus\{0\}.
\label{eq:IBP3}
\end{equation}
As for the test function $\phi$, we could think of the relations
\eqref{eq:IBP1}
and \eqref{eq:IBP3} as the fundamental relationship in place of
\eqref{eq:Phirel}
and \eqref{eq:Phirel2}. This allows us to place conditions on the derivative
$\phi'$ in place of $\phi$. For our considerations, we need one more piece of
information:
\begin{equation}
\int_\R\phi'(\tau)\diff\tau=0,
\label{eq:IBP4}
\end{equation}
which is obvious for tempered test functions $\phi$.

\section{Main results}

\subsection{The setup}

We consider $\R^4$ only, and consider for $\psi\in L^1(\R)$ the
associated function
\begin{equation}
\Psi(x)=-\frac{1}{\imag\pi |x|^2}\int_\R\e^{\imag\pi\tau |x|^2}\psi(\tau)
\diff \tau,\qquad x\in\R^4\setminus\{0\}.
\label{eq:Psi1}
\end{equation}
This is the same as the relation \eqref{eq:IBP1} only $\psi$ replaces $\phi'$
while $\Psi$ replaces $\Phi$.
For real $\tau$, let $H_\tau$ denote the function
\begin{equation}
H_\tau(x):=\frac{\e^{\imag \pi|x|^2\tau}}{\imag\pi|x|^2},\qquad
x\in\R^4\setminus\{0\},
\label{eq:Hfun1}
\end{equation}
which is locally integrable and decays at infinity. As such, it is a
tempered distribution, and its Fourier transform equals
\begin{equation}
\hat H_\tau(y)=\frac{1-\e^{-\imag\pi|y|^2/\tau}}{\imag\pi|y|^2}
=\frac{1}{\imag\pi|y|^2}-H_{-1/\tau}(y).
\label{eq:Hfun2}  
\end{equation}
This is the integrated version of the Fourier transformation law for
Gaussians \eqref{eq:GL} in dimension $d=4$.
Indeed, if we differentiate with respect to $\tau$
in \eqref{eq:Hfun2}, we recover \eqref{eq:GL}. In other words,
differentiation with respect to $\tau$ gives us that
$\hat H_\tau+H_{-1/\tau}$ is independent of
$\tau$. By letting $\tau$ tend to $0$, the identification with the Newton
kernel as in \eqref{eq:Hfun2} follows from the Riemann-Lebesgue lemma. 
In view of \eqref{eq:Hfun2}, the Fourier transform of the function $\Psi$
given by \eqref{eq:Psi1} is in the sense of distribution theory
\begin{equation}
\hat\Psi(y)=-\int_\R\hat H_\tau(y)\psi(\tau)
\diff \tau=-\frac{1}{\imag\pi|y|^2}\int_\R\psi(\tau)\diff\tau
+\frac{1}{\imag\pi|y|^2}\int_\R
\e^{-\imag\pi |y|^2/\tau}\psi(\tau)\diff\tau,
\qquad y\in\R^4\setminus\{0\}.
\label{eq:Psi2}
\end{equation}
This formula extends \eqref{eq:IBP3}.

\subsection{Fourier uniqueness meets Heisenberg uniqueness and the
Klein-Gordon equation}
In \cite{HM1}, in the context of the Klein-Gordon equation in $1+1$
dimensions, Hedenmalm and Montes found discrete uniqueness sets along
characteristic directions, based on ideas from dynamical systems and
ergodic theory. We apply the approach in \cite{HM1}, \cite{HM2}, \cite{HM3},
and \cite{CHM} 
to obtain a uniqueness result for the pair $\psi,\Psi$ connected by
\eqref{eq:Psi1}.  
Let $H^1_+(\R)$ denote the Hardy space of the upper half-plane. It
may be defined as the subspace of functions in $L^1(\R)$ with Poisson harmonic
extension to $\Hyp$  which is holomorphic.

\begin{thm}
Let $\psi\in L^1(\R)$ and $\Psi$ be as above. If   
$\Psi(x)=\hat\Psi(y)=0$ holds for all $x,y\in\Z^4\setminus\{0\}$,
and if $\Psi(x)=\ordo(|x|^{-2})$ as $|x|\to0$, then $\psi\in H^1_+(\R)$
and, as a consequence, $\Psi(x)\equiv0$ on $\R^4\setminus\{0\}$.
\label{thm:main}
\end{thm}

\begin{proof}
In view of the assumption that $\Psi(x)=\ordo(|x|^{-2})$ as $|x|\to0$,
it follows from \eqref{eq:Psi1} that $\psi\in L^1(\R)$ annihilates the
constant function $1$.
Moreover, by the Lagrange (or Jacobi)
four squares theorem, each positive integer may be written as $|x|^2$ for
some $x\in\Z^4\setminus\{0\}$. Consequently, we see from \eqref{eq:IBP1} and
\eqref{eq:IBP3} that $\psi$ also annihilates the subspace of
$L^\infty(\R)$ spanned by the functions $\e^{\imag\pi m\tau}$ and
$\e^{-\imag \pi n/\tau}$, where $m,n\in\Z_+$ and $\tau$ is the real
variable. By Theorem 1.8.2 in \cite{HM2}, which relies on methods
developed in \cite{HM3} and is motivated by \cite{HM1},
we may conclude that $\psi\in H^1_+(\R)$. Finally,
in view of the standard Fourier analysis characterization of $H^1_+(\R)$, it
follows from this and \eqref{eq:IBP1} that $\Psi=0$ on $\R^4\setminus \{0\}$.
\end{proof}

We return to the initial setup with $\phi$ and $\Phi$. We think of $\phi'=\psi$
and $\Phi=\Psi$. Let $C_0(\R)$ denote
the space of continuous functions on $\R$ with limit value $0$ at infinity.
Then the condition at the origin in Theorem \ref{thm:main} may be replaced
by $\phi\in C_0(\R)$. 

\begin{cor}
Let $\Phi$ be given by \eqref{eq:IBP1},  where
$\phi\in C_0(\R)$ with $\phi'\in L^1(\R)$ and $d=4$. If 
$\Phi(x)=\hat\Phi(y)=0$ for all $x,y\in\Z^4\setminus\{0\}$,
then $\phi'\in H^1_+(\R)$
and, as a consequence, $\Phi(x)\equiv0$ on $\R^4\setminus\{0\}$. 
\end{cor}

\begin{rem}
The above theorem is a four-dimensional analogue of the uniqueness
part of the Fourier interpolation formula found by Radchenko and Viazovska
\cite{RV19}. That work in its turn was motivated by Fourier interpolation
formul\ae{} associated with optimizing the Cohn-Elkies method for sphere packing
\cite{Viaz17}, \cite{CKMRV}.
\end{rem}

\end{document}